\documentclass[reqno,12pt]{amsart}%
\usepackage[body={17cm,21cm}]{geometry}

\usepackage[T1]{fontenc}
\usepackage{graphicx}
\usepackage{mathrsfs}
\usepackage{amscd,latexsym,amsthm,amsfonts,amssymb,amsmath,amsxtra}
\usepackage[colorlinks, urlcolor=blue,  citecolor=blue]{hyperref}
\usepackage{enumerate}
\usepackage[all]{xy}%
\providecommand{\U}[1]{\protect\rule{.1in}{.1in}}
\RequirePackage{amsmath}
\RequirePackage{amssymb}

\newcommand{\BG}{{\mathbb {G}}}

\newcommand{\RR}{{\mathrm {R}}}

\newcommand{\Br}{{\mathrm{Br}}}

\newcommand{\inv}{{\mathrm{inv}}}

\newcommand{\Ker}{{\mathrm{Ker}}}

\newcommand{\Pic}{\mathrm{Pic}}

\renewcommand{\mod}{\ \mathrm{mod}\ }

\newcommand{\et}{{\operatorname{\acute{e}t}}}

\numberwithin{equation}{section}

\theoremstyle{plain}

\newtheorem{theorem}{Theorem}[section]
\newtheorem{lemma}[theorem]{Lemma}

\newtheorem{proposition}[theorem]{Proposition}
\theoremstyle{definition}

\newtheorem{remark}[theorem]{Remark}

\newtheorem{exam}{Example}

\begin{document}

\title[Integral points on twisted Markoff surfaces] {Integral points on twisted Markoff surfaces}

\author{Sheng Chen}
\address{School of Mathematical Science \\ Capital Normal University \\ Beijing 100048, China}
\email{chenshen1991@163.com}

\begin{abstract}
  We study the integral Hasse principle for affine varieties of the form
$$
  ax^{2}+y^{2}+z^{2}-xyz=m
$$
using Brauer-Manin obstruction, and we produce examples whose Brauer groups include $4$-torsion elements .
We use methods of [5] to describe them and in some cases we show that there is no
Brauer-Manin obstruction to the integral Hasse principle for them.

\end{abstract}

\maketitle
\section{Introduction}

In recent papers [4] and [10], Colliot-Th\'el\`ene , Wei, Xu, D. Loughran and V. Mitankin, studied the integral Hasse principle
and strong approximation for Markoff surfaces, using the Brauer-Manin obstruction.
 For Markoff surfaces , D. Loughran and V. Mitankin obtained the following beautiful result :\\

Assume that $m \in \Bbb Z$ is such that affine surface $\mathcal{U}_{m}$ defined by
$$x^2+y^2+z^2-xyz=m.$$
has a Brauer-Manin obstruction to the integral Hasse principle. Then
$$ m-4 \mod Q^{\times^2} \in \langle \pm 1,2,3,5\rangle \subset Q^{\times}/Q^{\times^2}.$$
As they pointed out , this can be seen as an analogue of the finiteness of exceptional spinor classes in the study of the representation of an integer by a
ternary quadratic form (see [3 ,\S 7]).\\

Now, fix $ m, a\in \Bbb Z ,m \ne 0 ,4a$. Let $U \subset {\Bbb A}^3_{\Bbb Z}$
be the affine scheme over $\Bbb Z$ defined by the equation
$$ax^2+y^2+z^2-xyz=m.$$
We study the Brauer-Manin obstruction to the integral Hasse principle for U.
In particular, we have similar results :
\begin{theorem}
Assume that  $[Q(\sqrt{a},\sqrt{m},\sqrt{m-4a}):Q]=8$, let $(a,m)=p_1^{n_1}p_2^{n_2} \cdots p_s^{n_s}$, where $(a,m)$ is the greatest common divisor of a and m, $p_i$ are prime for $1 \le i \leq s$ .
If there is a Brauer-Manin obstruction to the integral Hasse principle for U, we have
$$ m-4a \mod Q^{\times^2} \in \langle \pm 1,2,3,\{p_i\}_{1 \le i \leq s} \rangle \subset Q^{\times}/Q^{\times^2}.$$
\end{theorem}
Moreover, we will give examples whose Brauer groups include $4$-torsion elements,
and with some assumptions, we can show that there is no Brauer-Manin obstruction to the integral Hasse principle for them.
\\
 As noted in [11], an often used strategy for proving that a class $\mathcal A \in \Br(U)$ of order n gives no obstruction to the Hasse principle is to demonstrate the
existence of a finite place v of k such that the evaluation map $X(k_v) \to (\Br k_v)[n]$, sending a
point $P \in Y(k_v)$ to the evaluation $\mathcal A(P) \in (\Br k_v)[n]$, is surjective .However,the local invariant of non-cyclic algebra is difficult to compute in general. Based on ideas of [5],
we will construct explicit representatives of non-cyclic Brauer classes on affine surfaces, and compute its local invariants
in special places.

\bigskip

{\bf Notation} Let $k$ be a field and $\overline k$ a separable closure of $k$.
If $X$ is a k-variety,we write $\overline{X}=X \times_{k} \overline{k}$.
If $X$ is an integral $k$-variety, we let $k(X)$ denote the function field of $X$.
If $X$ is a geometrically integral $k$-variety, we let $\overline{k}(X)$ denote
the function field of $\overline{X}$.
We  let $\Pic(X)=
H^1_{\et}(X,\BG_{m})$ denote the Picard group of a scheme $X$. We let
$\Br(X)=H^2_{\et}(X,\BG_{m})$ denote the Brauer group of  a scheme $X$.
If $X$ is a regular integral $k$-variety, the natural map $$\Br(X) \to \Br(k(X))$$
is injective. We let $$\Br_{1}(X) = \Ker [\Br(X) \to \Br({\overline X})]$$ denote the algebraic Brauer group of a k-variety.

\section{Algebraic Brauer group of cubic surface}
Follow J.-L.Colliot-Th\'el\`ene,Dasheng Wei,and Fei Xu, we have

\begin{lemma}
  Let $ X   \subset \mathbb{P}_k^3$ defined by equation
$$ ax^{2}t+y^{2}t+z^{2}t-xyz=mt^{3}$$
over a field k of characteristic zero with $a\in k^{\times},a\notin k^2, $ then X is smooth if and only if $m\ne0 ,4a$.In this case,the 27 lines in X are defined over k($\sqrt{a},\sqrt{m},\sqrt{m-4a}$) by the following equations:
$$H_1:x=t=0; \qquad H_2:y=t=0 ;\qquad H_3:z=t=0$$ \\
and
$$
\left\{\begin{array}{ccc}
l_1(\varepsilon,\delta):& x=2\varepsilon t,& y-\varepsilon z=\delta \sqrt{m-4a}t \\
l_2(\varepsilon,\delta):& y=2\varepsilon \sqrt{a}t,&z-\varepsilon \sqrt{a}x=\delta \sqrt{m-4a}t\\
l_3(\varepsilon,\delta):& z=2\varepsilon \sqrt{a}t,&\sqrt{a}x-\varepsilon y=\delta \sqrt{m-4a}t\\
l_4(\varepsilon,\delta):& \sqrt{a}x=\varepsilon \sqrt{m}t,&\sqrt{a}y=\frac{1}{2} (\varepsilon \sqrt{m}+\delta\sqrt{m-4a})z\\
l_5(\varepsilon,\delta):& y=\varepsilon \sqrt{m}t,&z=\frac{1}{2}(\varepsilon \sqrt{m}+\delta\sqrt{m-4a})x\\
l_6(\varepsilon,\delta):& z=\varepsilon \sqrt{m}t,&x=\frac{1}{2a}(\varepsilon \sqrt{m}+\delta\sqrt{m-4a})y\\
\end{array}
\right
.$$

with\quad $\varepsilon=\pm1$, and $\delta=\pm1$.Moreover the intersection number\\
$$ (l_i(1,1),l_j(1,1))=0 \qquad whenever \ \ 1 \le i\neq j\leq6. $$
\end{lemma}

\begin{proof}
The results follow from straight forward computation.
\end{proof}.\\

\begin{proposition}
Let $ X   \subset \mathbb{P}_k^3$ defined by equation
$$ ax^{2}t+y^{2}t+z^{2}t-xyz=mt^{3}$$
over a field k of characteristic zero with $a\in k^{\times},a\notin k^2$,and  $m\ne0 ,4a$.Then $\Br(X)/\Br_0(X)=0$ or $ \Br(X)/\Br_0(X)\cong \mathbb{Z}/2 $ with generator $((\frac{x}{t})^2-4,m-4a)$.
\end{proposition}
\begin{proof}
The proof is completely similar to [4,Proposition3.2] , for later application,we only give computations of $\Br(X)/\Br_0(X)$ in somes cases.
Since X is geometrically rational,we have $\Br(X)=\Br_1(X)$.\ Since$ X(k) \neq \emptyset$, we have the following isomorphism
$$ \Br_1(X)/\Br_0(X) \cong H^1(k,\Pic(\overline{X})) $$
by the Hochschild-Serre spectral sequence. By [7,Chapter V,Proposition4.8],there is $ l\in \Pic(\overline{X})$ satisfying the following intersecton property
\begin{center}
$(l,l)=1 \quad (l,l_i(1,1))=0 $\quad for $1 \leq i\leq6$.\\
\end{center}
such that $\{ H_i(1,1): \ 1\leq i\leq 6\}\cup \{ l\} $ forms a basis of $\Pic(\overline{X})$.\\
Since $$ (H_j, l_i(1,1)) =\begin{cases} 1 \ \ \ & \text{$i-j \equiv 0$ or $3 \mod 6$} \\
0 \ \ \ & \text{otherwise} \end{cases} $$

where $1\leq j\leq3,1\leq i\leq 6$. One concludes that
$$ H_j=l-l_j(1,1)-l_{j+3}(1,1)$$
in $\Pic(\overline{X})$ for $ 1\leq j\leq3$ by [7,Chapter V,Proposition 4.8(e)].
For simplicity, we write $l_i$ for $l_i(1,1)$ with $1\leq i\leq 6$ . If $[k(\sqrt{a},\sqrt{m},\sqrt{m-4a}):k]=2$, there exists $\sigma \in Gal(k(\sqrt{a})/k)$ such that $\sigma (\sqrt{a}) =-\sqrt{a}$ .\\
1.$\sqrt{m} \in k,\sqrt{m-4a} \in k$, we have
$$
\begin{cases}
\sigma (l_1)=l_1 , \sigma (l_2)=l-l_{3}-l_4\\
\sigma (l_3)=l-l_{2}-l_4 , \sigma (l_4)=l-l_{2}-l_3 \\
\sigma (l_5)=l_5 ,   \sigma (l_6)=l_6\\
\sigma (l) =2l-l_{2}-l_3-l_4\\
\end{cases}
$$
Since $Ker(1+\sigma)=(l-l_2-L_3-l_4)$, \ \ $Im(\sigma-1)=(l-l_2-L_3-l_4)$,
we have  $H^1(k,Pic(\overline{X}))=0$ .\\

2. $\sqrt{m} \in k,\sqrt{m-4a} \notin k$, we have
$$
\begin{cases}
\sigma (l_1)= 2l-l_1-l_2-l_3-l_5-l_6  \\
\sigma (l_2)= l-l_1-l_6 \\
\sigma (l_3)= l-l_1-l_5 \\
\sigma (l_4)= l-l_5-l_6  \\
\sigma (l_5)= 2l-l_1-l_3-l_4-l_5-l_6 \\
\sigma (l_6)= 2l-l_1-l_2-l_4-l_5-l_6\\
\sigma  (l) = 4l-2l_1-2l_5-2l_6-l_2-l_3-l_4 \\
\end{cases}
$$
Since
\begin{center}
    \ \ $Ker(1+\sigma)=(l-l_1-l_6-l_2,l-l_1-l_5-l_3,l-l_4-l_5-l_6),$\\
     $Im(\sigma-1)=(l-l_1-l_6-l_2,l-l_1-l_5-l_3,l-l_4-l_5-l_6),$\\
\end{center}
we have $H^1(k,Pic(\overline{X}))=0$ .\\

3. $\sqrt{m} \notin k,\sqrt{m-4a} \in k$, we have
$$\begin{cases}
\sigma (l_1)=l_1  \\
\sigma (l_2)=l-l_3-l_4 \\
\sigma (l_3)=l-l_2-l_4 \\
\sigma (l_4)=2l-l_2-l_3-l_4-l_5-l_6\\
\sigma (l_5)=l-l_4-l_6  \\
\sigma (l_6)=l-l_4-l_5  \\
\sigma  (l) = 3l-l_2-l_3-2l_4-l_5-l_6 \\
\end{cases}
$$
Since
\begin{center}
$Ker(1+\sigma)=(l-l_2-l_3-l_4,l-l_4-l_5-l_6),$\\
 $Im (\sigma-1)=(l-l_2-l_3-l_4,l-l_4-l_5-l_6),$\\
\end{center}
we have $ H^1(k,Pic(\overline{X}))=0$.\\

4.$\sqrt{m} \notin k,\sqrt{m-4a}\notin k$ , we have
$$
\begin{cases}
\sigma (l_1)=2l-l_1-l_2-l_3-l_5-l_6 \\
\sigma (l_2)=l-l_1-l_6 , \sigma (l_3)=l-l_1-l_5 \\
\sigma (l_4)=l_4  , \sigma (l_5)=l-l_1-l_3 \\
\sigma (l_6)=l-l_1-l_2 \\
\sigma  (l) = 3l-2l_1-l_2-l_3-l_5-l_6 \\
\end{cases}
$$
Since \begin{center}
$Ker(1+\sigma)=(l-l_1-l_2-l_6,l-l_1-l_3-l_5),$\\
$Im (\sigma-1)=(l-l_1-l_2-l_6,l-l_1-l_3-l_5),$\\
\end{center}
we have $ H^1(k,Pic(\overline{X}))=0 $.
\end{proof}

\begin{proposition}
  Let U be the affine variety over a field of characteristic zero defined by the equation
  $$ ax^{2}+y^{2}+z^{2}-xyz=m$$
  where $a\in k^{\times},a\notin k^2$, $m\ne0 ,4a$. If $[k(\sqrt{a},\sqrt{m},\sqrt{m-4a}):k]=8$, we have\\
  $$ Br_1(U)/Br_0(U) \cong \mathbb{Z}/2\oplus\mathbb{Z}/2$$
  with generators $(x-2,m-4a)$,$(x+2,m-4a)$
\end{proposition}
\begin{proof}
  Let $G=Gal(k(\sqrt{a},\sqrt{m},\sqrt{m-4a})/k)$, there exist $\sigma$,$\tau$ and $\theta$ $\in$ G,such that
  \begin{center}
  $\sigma(\sqrt{a})=-\sqrt{a},\quad \sigma(\sqrt{m})=\sqrt{m},\ \ \sigma(\sqrt{m-4a})=\sqrt{m-4a}.$\\
  $\tau(\sqrt{a})=-\sqrt{a},  \quad   \tau(\sqrt{m})=-\sqrt{m},\ \ \tau(\sqrt{m-4a})=\sqrt{m-4a}.$\\
  $\theta(\sqrt{a})=-\sqrt{a},\quad \theta(\sqrt{m})=\sqrt{m},\ \ \theta(\sqrt{m-4a})=-\sqrt{m-4a}.$\\
  \end{center}
By [4,Proposition 2.2], $Pic(\overline{U})$ is given by the following quotient group
  $$ (\oplus^{6}_{i=1}\mathbb{Z}l_{i}\oplus \mathbb{Z}l)/(l-l_j-l_{j+3}:1\le j \le3)\cong \oplus^4_{i=1}\mathbb{Z}{\overline l_i}. $$
By computations in proposition 2.2, we have \\
\begin{enumerate}[(i)]
\item $\begin{cases}
\sigma (\overline l_1)=\overline l_1\\
\sigma (\overline l_2)=\overline l_1-\overline l_{3}\\
\sigma (\overline l_3)=\overline l_1-\overline l_{2} \\
\sigma (\overline l_4)=\overline l_1+\overline l_4-\overline l_{2}-\overline l_3\\
\end{cases}$ \\
Since $Ker(1+\sigma)=(\overline l_1-\overline l_{3}-\overline l_{2})$ , $Im(\sigma-1)=(\overline l_1-\overline l_{3}-\overline l_{2})$,\\
we have $H^1(\langle \sigma \rangle, Pic(\overline{U}))=0$.
\item  $\begin{cases}
  \theta (\overline l_1)=-\overline{l_1}  \\
  \theta (\overline l_2)= \overline l_3-\overline l_1 \\
  \theta (\overline l_3)= \overline l_2-\overline l_1 \\
  \theta (\overline l_4)= \overline l_2+\overline l_3-\overline l_1-\overline l_4\\
\end{cases}$\\

\item $\begin{cases}
  \tau (\overline l_1)=\overline l_1      \\
  \tau (\overline l_2)=\overline l_1-\overline l_3\\
  \tau (\overline l_3)=\overline l_1-\overline l_2 \\
  \tau (\overline l_4)=-\overline l_4\\
\end{cases}$\\
\end{enumerate}
  Let $H={\langle \tau,\theta \rangle}$, we have the following exact sequence
  \begin{center}
  $0\to H^1(H,Pic(\overline{U})^{\langle \sigma \rangle})\to H^1(G,Pic(\overline{U}))\to H^1(\langle \sigma \rangle,Pic(\overline{U}))=0$.
  \end{center}
  where $Pic(\overline{U})^{\langle \sigma \rangle}=(\overline l_1,\overline l_2-\overline l_4,\overline l_3-\overline l_4)$. Let us compute $H^1(H,Pic(\overline{U})^{\langle \sigma \rangle})$,
  we have the following exact sequence
  \begin{center}
  $0\to H^1(\langle \theta \rangle,Pic(\overline{U})^{\langle \sigma,\tau \rangle})\to H^1(H,Pic(\overline{U})^{\langle \sigma \rangle})\to H^1(\langle \tau \rangle,Pic(\overline{U})^{\langle \sigma \rangle})$.
  \end{center}
  Since $$ \begin{cases}
  \tau(\overline l_1)=\overline l_1 \\
  \tau(\overline l_2-\overline l_4)=\overline l_1-\overline l_3+l_4\\
  \tau(\overline l_3-\overline l_4)=\overline l_1-\overline l_2+l_4\\
  \end{cases}
  $$
  we have
  $$ Ker(1+\tau)=(\overline l_1-\overline l_3-\overline l_2+2\overline l_4), Im(\tau-1)=(\overline l_1-\overline l_3-\overline l_2+2\overline l_4).$$
  One concludes that  $H^1(\langle \tau \rangle,Pic(\overline{U})^{\langle \sigma \rangle})=0$, hence
  $$ H^1(G,Pic(\overline{U}))\cong H^1(H,Pic(\overline{U})^{\langle \sigma \rangle})\cong H^1(\langle \theta \rangle,Pic(\overline{U})^{\langle \sigma,\tau \rangle})$$
  where $Pic(\overline{U})^{\langle \sigma,\tau \rangle}=(\overline l_1,\overline l_2-\overline l_3)$.
  Since $$ \begin{cases}
  \theta(\overline l_1)=-\overline l_1 \\
  \theta(\overline l_2-\overline l_3)=\overline l_3-\overline l_2\\
\end{cases}
  $$
one has
  $$ H^1(\langle \theta \rangle,Pic(\overline{U})^{\langle \sigma,\tau \rangle})\cong \mathbb{Z}/2\oplus\mathbb{Z}/2. $$
  We obtain
  $$ H^1(G,Pic(\overline{U}))\cong \mathbb{Z}/2\oplus\mathbb{Z}/2. $$

 Note that $ax^{2}+y^{2}+z^{2}-xyz=m $ is equivalent to
 $$ (2z-xy)^2-4(m-4a)=(x^2-4)(y^2-4a)$$
 Arguing in the same way as in the proof of [4,Theorem 3.4], one obtains the generators $(x-2,m-4a)$ and $(x+2,m-4a)$.
 Indeed  since
 $$ \{x\pm 2=0\} \cap \{((x\mp2)(y^2-4a)=0 \} $$ is a closed subset of codimension $\geq 2$ on $U$, one obtains that $(x\pm 2, m-4a)\in \Br_1(U)$.
 This implies that
 $$ B=(x^2-4,m-4a)=(y^2-4a,m-4a)=(z^2-4a,m-4a)  \in \Br_1(U).  $$
 Now we show that $B$ is not constant.
 $$ \pi:  \ U\rightarrow \Bbb A^1; \ (x, y, z) \mapsto x . $$ The generic fibre $U_\eta \xrightarrow{\pi_\eta} \eta$ induces $$\pi_{\eta}^*: \Br(\eta) \rightarrow \Br(U_\eta) \ \ \ \text{with} \ \ \  \ker (\pi_\eta^*) =(x^2-4, m-ax^2)$$ by [6,Theorem 5.4.1].
 Since $[k(\sqrt{a},\sqrt{m},\sqrt{m-4a}) : k]=8$,  the residue of $(x^2-4,m-4a)$ at $(m-ax^2)$ is different from that of $(x^2-4,m-ax^2 )$.
 This implies that $\pi_\eta^*(x^2-4,m-4a)$ is not constant by the Faddeev exact sequence. Since $\pi_\eta^*(x^2-4,m-4a)$ is the pull-back of $B$ by the projection map $U_\eta\rightarrow U$, one concludes that $B$ is not constant.
\end{proof}

\bigskip
\section {Examples of Brauer-Manin obstruction}
We now give examples of Brauer-Manin obstruction to the integral Hasse principle. Here the results are inspired by the results in [10,\S 5.3,\S 5.4]
\begin{lemma}
If $p$ is an odd prime with $(p,m-4a)=1$, then the following elements
$$ (x+2,m-4a),(x-2,m-4a),(z^2-4a,m-4a),(y^2-4a,m-4a)$$
vanish over $U(\mathbb{Z}_p)$. If $m-4a > 0$,these elements vanish over $U(\mathbb{R})$.
In particular, if $a < 0$, $(x^2-4,m-4a)_\infty=(z^2-4a,m-4a)_\infty=(y^2-4a,m-4a)_\infty=0$.

\end{lemma}
\begin{proof}
Arguing in the same way as in the proof of [4,Lemma 5.1 ], one can easily verify this.
\end{proof}
\begin{lemma}
Let $p \mid (m-4a)$ be odd , if $p \nmid a$,  any singular point T$(x,y,z) \in U(\Bbb F_{p})$ satisfies
$$ x^2=4,\quad y^2=4a$$
\end{lemma}
\begin{proof}
 Since $T \in U(\Bbb F_{p})$ is singular, we have
 $$\begin{cases}
 (2z-xy)^2=(x^2-4)(y^2-4a)\\
 2ax-yz=0\\
 2y-xz=0\\
 2z-xy=0\\
\end{cases}
 $$
 We obtain $x^2=4,\quad y^2=4a$.
\end{proof}
\begin{lemma}
Let $p \geq 3$ such that $p\mid (m-4a)$ and $p\nmid a$, with $ord_p(m-4a)$ even but $m-4a \notin {Q_p^\times}^2$.
Let $\mathcal B_1=(x^2-4,m-4a), \  \mathcal B_2=(x+2,m-4a).$ \ For all $T \in U(\mathbb{Z}_p)$, we have \\
If $(\frac{a}{p})=-1$,
$$ \{inv_{p}\mathcal B_1(T),inv_{p}\mathcal B_2(T)\} = \{{0,0}\} .$$
If $(\frac{a}{p})=1$,
$$ \{inv_{p}\mathcal B_1(T),inv_{p}\mathcal B_2(T)\} \in \{\{{0,0}\},\{\frac{1}{2},\frac{1}{2}\},\{\frac{1}{2},0\}\} .$$
\end{lemma}
\begin{proof}
Note that $ax^{2}+y^{2}+z^{2}-xyz=m $ is equivalent to
$$ (2z-xy)^2=(x^2-4)(y^2-4a)+4(m-4a) $$
As $(m-4a) \notin {Q_p^\times}^2$ and $ord_p(m-4a)$ is even, it follows that $ord_p((x^2-4)(y^2-4a))$ is even.\\
If $ord_p(x^2-4)$ is even ,one obtains $ord_p(x+2)$ is even,\ hence $\{inv_{p}\mathcal B_1(T),inv_{p}\mathcal B_2(T)\}=\{\{0,0\}\}$.
Assume that $ord_p(x^2-4)$ is odd, then $ord_p(y^2-4a)$ is odd, thus $y^2 \equiv 4a \bmod p$. If $(\frac{a}{p})=-1$, this is a contradiction.
Now let $(\frac{a}{p})= 1$, if $ord_p(x+2)$ is odd, we obtain
$$ \{inv_{p}\mathcal B_1(T),inv_{p}\mathcal B_2(T)\}=\{\frac{1}{2},\frac{1}{2}\}. $$
If $ord_p(x-2)$ is odd, we obtain
$$ \{inv_{p}\mathcal B_1(T),inv_{p}\mathcal B_2(T)\}=\{\frac{1}{2},0\}.$$
We now show that these possibilities can be realised. We first consider $(\frac{a}{p})=-1$,
we have smooth point $(2,0,0) \in U(\Bbb F_{p})$ by lemma 3.2, hence $ U(\mathbb{Z}_p) \ne \emptyset$.
Now we consider $(\frac{a}{p})=1$, we have smooth point $(0,2\sqrt{a},0)\in U(\Bbb F_{p})$ by lemma 3.2, hence there exists
$T \in U(\mathbb{Z}_p)$ such that $$ \{inv_{p}\mathcal B_1(T),inv_{p}\mathcal B_2(T)\} = \{\{0,0\}\}.$$
Suppose $ord_p(m-4a)=2$, let $ s \in \Bbb F_{p}^{\times}$ such that $(\frac{\frac{m-4a}{p^2}-s}{p})=1$. Let $s' \in \mathbb{Z}_p$ such that
$s' \equiv s \bmod p$ . There is $y_0 \in \mathbb{Z}_p$ such that ${y_0}^2-4a=ps'$ by Hensel lemma. Let $x_0=p-2$, we consider the following
equation
$$ p^2t^2-4(m-4a)=({x_0}^2-4)({y_0}^2-4a)$$
over $\mathbb{Z}_p$.
That is  $$ t^2-\frac{4(m-4a)}{ p^2}=(p-4)s'.$$
By Hensel lemma, one can see that the equation has solutions. Let $t_0$ denote one of the solutions and let $z_0 \in \mathbb{Z}_p$ such that
$2z_0-x_0y_0=pt_0$, then $T_0=(x_0,y_0,z_0) \in U(\mathbb{Z}_p)$, we have
$$\{inv_{p}\mathcal B_1(T_0),inv_{p}\mathcal B_2(T_0)\}=\{\{\frac{1}{2},\frac{1}{2}\}\}.$$
If we let $x_1=p+2$, one can see that there exists $T_1=(x_1,y_1,z_1) \in U(\mathbb{Z}_p)$, we have
$$\{inv_{p}\mathcal B_1(T_1),inv_{p}\mathcal B_2(T_1)\}=\{\{\frac{1}{2},0\}\}.$$
For $ord_p(m-4a)>2$, the proof is similar .

\end{proof}

\begin{proposition}
	Suppose conditions of Proposition 2.3 are satisfied, if there exists $p \geq 5$ such that $p\nmid a$ and $ord_p(m-4a)$ is odd,
there is no Brauer--Manin obstruction to integral Hasse principle.
\end{proposition}
\begin{proof}
We can assume that $U(A_\mathbb Z) \ne \emptyset$, where $A_\mathbb Z=\RR \times \prod_p \mathbb{Z}_p$ ,otherwise there is nothing to prove.
Let $$ \begin{cases}
\mathcal B_1 =(x^2-4,m-4a)=(y^2-4a,m-4a)=(z^2-4a,m-4a) \\
\mathcal B_2=(x+2,m-4a)\\
\end{cases}
$$
to prove the proposition ,
it suffices to show for all $(\varepsilon_1,\varepsilon_2) \in (\mathbb{Z}/2\mathbb{Z})^{2}$, there exists
$\zeta \in U(\mathbb{Z}_p)$ such that
$$ (inv_{p}\mathcal B_1(\zeta),inv_{p}\mathcal B_2(\zeta))=(\varepsilon_1,\varepsilon_2)$$
We first consider the case $(\frac{a}{p})=1$.\\
Since $p \geq 5$,there exist  $t,s \in \Bbb F_{p}$, such that the Legendre symbol $(\frac{t^2-4a}{p})=1$, $ (\frac{s^2-4a}{p})=-1$, let
$$ \upsilon_1=(2,t,t),\quad \upsilon_2=(2,s,s). $$
One can see that $\upsilon_1$ and $\upsilon_2$ are smooth points of $U(\Bbb F_{p})$ by lemma 3.2, hence there exist $\mu_{1},\mu_{2} \in U(\mathbb{Z}_p)$
such that $\mu_{i}\equiv \upsilon_{i} \bmod p$ for $1 \leq i\leq 2$, we obtain
$$ \begin{cases}
(inv_{p}\mathcal B_1(\mu_1),inv_{p}\mathcal B_2(\mu_1))=(0,0)\\
(inv_{p}\mathcal B_1(\mu_2),inv_{p}\mathcal B_2(\mu_2))=(\frac{1}{2},0)\\
\end{cases}
$$
Since $p \geq 5$ ,we can choose $e,f \in \Bbb F_{p}$ such that the Legendre symbol $(\frac{e}{p})=-1$,$(\frac{e^2-4e}{p})=1$,
$(\frac{f}{p})=-1$, $(\frac{f^2-4f}{p})=-1$. Let

$$ \upsilon_3=(e-2,2\sqrt{a},(e-2)\sqrt{a}),\quad \upsilon_4=(f-2,2\sqrt{a},(f-2)\sqrt{a}) .$$
One can see that $\upsilon_3$ and $\upsilon_4$ are smooth points of $U(\Bbb F_{p})$ by lemma 3.2, hence there exist $\mu_{3},\mu_{4} \in U(\mathbb{Z}_p)$
such that $\mu_{i}\equiv \upsilon_{i} \bmod p$ for $3 \leq i\leq 4$, we obtain
$$ \begin{cases}
(inv_{p}\mathcal B_1(\mu_3),inv_{p}\mathcal B_2(\mu_3))=(0,\frac{1}{2})\\
(inv_{p}\mathcal B_1(\mu_4),inv_{p}\mathcal B_2(\mu_4))=(\frac{1}{2},\frac{1}{2})\\
\end{cases}
$$
Now assuing $(\frac{a}{p})=-1$.
Let $\eta_1=\mu_1,\eta_2=\mu_2$, we have\\
$$\begin{cases}
(inv_{p}\mathcal B_1(\eta_1),inv_{p}\mathcal B_2(\eta_1))=(0,0)\\
(inv_{p}\mathcal B_1(\eta_2),inv_{p}\mathcal B_2(\eta_2))=(\frac{1}{2},0)\\
\end{cases}
$$
Let
$$ \xi_3=(e-2,\alpha,\alpha) \quad  \xi_4=(f-2,\beta,\beta) $$
where $\alpha^2=ae, \beta^2=af $,one can see that $\xi_3$ and $\xi_4$ are smooth points of $U(\Bbb F_{p})$ by lemma 3.2, hence there exist $\eta_{3},\eta_{4} \in U(\mathbb{Z}_p)$
such that $\eta_{i}\equiv \xi_{i} \bmod p$ for $3 \leq i\leq 4$, we obtain
$$ \begin{cases}
(inv_{p}\mathcal B_1(\eta_3),inv_{p}\mathcal B_2(\eta_3))=(0,\frac{1}{2})\\
(inv_{p}\mathcal B_1(\eta_4),inv_{p}\mathcal B_2(\eta_4))=(\frac{1}{2},\frac{1}{2})\\
\end{cases}
$$
The proposition is established.
\end{proof}
As corollary, we obtain Theorem 1.1 .

\begin{lemma}
Given $ a,m \in \mathbb{Z} $, the equation $ax^2+y^2+z^2-xyz=m$
has solutions in $(\mathbb{Z}_p)^3$ for all primes p except for the following two cases:
\begin{enumerate}[(i)]
\item $a \equiv 1 \mod 4,  m \equiv 3 \mod 4 $.
\item $a \equiv 1 \mod 3,  m \equiv \pm 3 \mod 9$.\\
\end{enumerate}
\end{lemma}

\begin{proof}
We break up the proof into several cases. Let \\
$$ \begin{cases}
f=ax^2+y^2+z^2-xyz-m\\
f_x=2ax-yz\\
f_y=2y-xz\\
f_z=2z-xy\\
\end{cases}
$$
For $p>3$,we have
\begin{enumerate}[(i)]
\item $p\nmid a,p\mid m$. If $p\equiv1\mod4$, there exists $(0,t,s)\in(\Bbb F_{p})^3$, such that $f(0,t,s)\equiv 0 \mod p,
f_y(0,t,s)\not \equiv 0 \mod p$, hence f has a zero in $(\mathbb{Z}_p)^3$.\\
If $p\equiv 3 \mod 4$,we first consider$(\frac{a}{p})=-1$,there exists $(1,0,t)\in(\Bbb F_{p})^3$,such that $f(1,0,t)\equiv 0 \mod p,
f_z(1,0,t)\not \equiv 0 \mod p$, hence f has a zero in $(\mathbb{Z}_p)^3$. Now suppose $(\frac{a}{p})=1$,
there exists $(3,t,2t)\in(\Bbb F_{p})^3$,such that $f(3,t,2t)\equiv 0 \mod p,f_y(3,t,2t)\not \equiv 0 \mod p$,
hence f has a zero in $(\mathbb{Z}_p)^3$.

\item $p\nmid a, p\nmid m$. If $(\frac{m}{p})=-1$, since $p>3$, there exists $(0,t,s)\in(\Bbb F_{p})^3$,such that $f(0,t,s)\equiv 0 \mod p,
f_y(0,t,s)\not \equiv 0 \mod p$, hence f has a zero in $(\mathbb{Z}_p)^3$.\\
If $(\frac{m}{p})=1$,there exists $(0,t,0)\in(\Bbb F_{p})^3$,such that $f(0,t,0)\equiv 0 \mod p,f_y(0,t,0)\not \equiv 0 \mod p$,
hence f has a zero in $(\mathbb{Z}_p)^3$.
\item $p\mid a, p\mid m$. There exists $(2,1,1)\in(\Bbb F_{p})^3$, such that $f(2,1,1)\equiv 0 \mod p,
f_x(2,1,1)\not \equiv 0 \mod p$, hence f has a zero in $(\mathbb{Z}_p)^3$.
\item $p\mid a, p\nmid m$. An argument similar to (ii), one can prove f has a zero in $(\mathbb{Z}_p)^3$.
\end{enumerate}
For $p=3$, we have
\begin{enumerate}[(i)]
\item $3 \nmid a,3\mid m$. If $a \equiv 1\mod 3$, f has only zero $(0,0,0) \in (\Bbb F_{3})^3$. So f has no zeros in $(\mathbb{Z}_3)^3$
when $m \equiv \pm 3 \mod 9$.
Now assume $9 \mid m$,one can easily check the equation $ax^2+y^2+z^2-3xyz=\frac{m}{9}$ has a solution $(x_0,y_0,z_0) \in (\mathbb{Z}_3)^3$ using Hensel lemma.
Hence $f(3x_0,3y_0,3z_0)=0$.\\
If $a \equiv 2 \mod 3$, since $f(1,1,0)\equiv 0 \mod 3,f_y(1,1,0)\not \equiv 0 \mod 3$, f has a zero in $(\mathbb{Z}_3)^3$.
\item $3 \nmid a,3\nmid m$. If $m \equiv 1 \mod 3$, since $f(0,1,0)\equiv 0 \mod 3,f_y(0,1,0)\not \equiv 0 \mod 3$,
f has a zero in $(\mathbb{Z}_3)^3$.\\
If $m \equiv 2 \mod 3$, since $f(0,1,1)\equiv 0 \mod 3,f_y(0,1,1)\not \equiv 0 \mod 3$,f has a zero in $(\mathbb{Z}_3)^3$.
\item $3 \mid a,3\mid m$. Since $f(2,1,1)\equiv 0 \mod 3,f_x(2,1,1)\not \equiv 0 \mod 3$,f has a zero in $(\mathbb{Z}_3)^3$.
\item $3 \mid a,3\nmid m$. If $m \equiv 1 \mod 3$, since $f(0,1,0)\equiv 0 \mod 3,f_y(0,1,0)\not \equiv 0 \mod 3$,
f has a zero in $(\mathbb{Z}_3)^3$.\\
If $m \equiv 2 \mod 3$, since $f(0,1,1)\equiv 0 \mod 3,f_y(0,1,1)\not \equiv 0 \mod 3$, f has a zero in $(\mathbb{Z}_3)^3$.
\end{enumerate}
For $p=2$, we have\\
\begin{enumerate}[(i)]
\item $2 \nmid a,2\mid m$. Since $f(1,1,1)\equiv 0 \mod2,f_y(1,1,1)\not \equiv 0 \mod 2$, f has a zero in $(\mathbb{Z}_2)^3$.
\item $2 \nmid a,2\nmid m$. If $a\equiv m\mod8$, since $f(1,0,0)\equiv 0 \mod 8,ord_{2}(f_x(1,0,0))=1$, f has a zero in $(\mathbb{Z}_2)^3$.
If $a\equiv m-4\mod8$, since $f(1,2,0)\equiv 0 \mod8, ord_{2}(f_x(1,2,0))=1$, f has a zero in $(\mathbb{Z}_2)^3$.\\
If $a\equiv 3\mod4, m\equiv 1\mod4$, we first consider $m\equiv 1\mod8$, since $f(0,1,0)\equiv 0 \mod 8, ord_{2}(f_y(0,1,0))=1$,
f has a zero in $(\mathbb{Z}_2)^3$. Now suppose $m\equiv 5\mod8$, since $f(0,1,2)\equiv 0 \mod 8, ord_{2}(f_y(0,1,2))=1$,
f has a zero in $(\mathbb{Z}_2)^3$.\\
If $a\equiv 1\mod4, m\equiv 3\mod4$, note that all solutions of $f\equiv 0\mod2$ are $(1,0,0),(0,1,0)$ and $(0,0,1)$.
If we take for $(1,0,0)$, this implies $a\equiv m\mod4$,a contradiction. If we take for $(0,1,0)$ or $(0,0,1)$, this implies $m\equiv 1\mod4$, a contradiction. So f has no zero in $(\mathbb{Z}_2)^3$ in this case.
\item $2 \mid a, 2\mid m$. Since $f(0,1,1)\equiv 0 \mod 2,f_x(1,1,1)\not \equiv 0 \mod 2$, f has a zero in $(\mathbb{Z}_2)^3$.
\item $2 \mid a, 2\nmid m$. Since $f(1,1,1)\equiv 0 \mod 2,f_x(1,1,1)\not \equiv 0 \mod 2$, f has a zero in $(\mathbb{Z}_2)^3$.
\end{enumerate}
\end{proof}

\begin{proposition}
  Let $U$ be the scheme over $\Bbb Z$ given by
  \begin{equation}   ax^2+y^2+z^2-xyz=4a+2d^2  \end{equation}
  where $a,d$ are odd integers such that $(a,d)=1$, $3 \nmid (a-1)$, $\sqrt{a} \notin Q$,
$p \equiv \pm 1 \mod 8 $ or $(\frac{a}{p})=-1$ for $p\mid d$.
Then there is a Brauer-Manin obstruction to the integral Hasse principle for U.
\end{proposition}
\begin{proof}
By lemma 3.5, we have $U(A_\mathbb Z) \ne \emptyset$ .
Let $$\mathcal B=(x^2-4,2)=(z^2-4a,2)=(y^2-4a,2)$$
 we will show that for each point $T \in U(\mathbb{Z}_p)$, we have
\begin{equation}
  \inv_p \mathcal B(T) =
  \begin{cases}
    1/2 &\mbox{if } p = 2, \\
    0  &\mbox{otherwise,}
  \end{cases}
\end{equation}
If $p \nmid 2d^2$ the claim follows from Lemma 3.1 . If $p \mid d$ and $p \equiv \pm 1 \bmod 8$ ,we have $2 \in {Q_p^\times}^2$.
Thus $2d^2 \in {Q_p^\times}^2$.If $p \mid d$ and $(\frac{a}{p})=-1$, the claim follows from lemma 3.3. Finally, since $m - 4a > 0$,
the claim is trivial for $p=\infty $ . It remains to examine $p = 2$.\\
Assme now $p=2$. Let $T \in U(\mathbb{Z}_2)$, one easily see that there is at least one coordinate of T belonging to ${\mathbb{Z}_2^\times}$.
A simple Hilbert symbol calculation implies the claim for $p=2$.\\
\end{proof}

\begin{proposition}
Let $U$ be the scheme over $\Bbb Z$ given by
\begin{equation}   ax^2+y^2+z^2 -xyz=4a+3d^2  \end{equation}
where a is an even integer such that $a \equiv 1 \mod 3$ and $\sqrt{a} \notin Q$,
$ p \equiv \pm 1  \mod 12$ or $(\frac{a}{p})=-1$ for $p\mid d$.
When $\sqrt{4a+3d^2} \notin Q $, there is a Brauer-Manin obstruction to the integral Hasse principle for U.
\end{proposition}
\begin{proof}
One can see $U(A_\mathbb Z) \ne \emptyset$ by lemma 3.5.\\
Let $\mathcal B=(x^2-4,3)=(z^2-4a,3)=(y^2-4a,3)$, we will show that for each point $T \in U(\mathbb{Z}_p)$, we have
\begin{equation}
  \inv_p \mathcal B(T) =
  \begin{cases}
    1/2 &\mbox{if } p = 3, \\
    0  &\mbox{otherwise,}
  \end{cases}
\end{equation}
so that $\mathcal B$  gives an obstruction to the Hasse principle.\\
If $p \nmid 6d^2$ the claim follows from Lemma 3.1 . If $p \mid d$ and $p \equiv \pm 1 \bmod 12$ ,we have $3 \in {Q_p^\times}^2$.
Thus $3d^2 \in {Q_p^\times}^2$.If $p \mid d$ and $(\frac{a}{p})=-1$, the claim follows from lemma 3.3. Finally, since $m - 4a > 0$,
the claim is trivial for $p=\infty $ . It remains to examine $p = 2, 3$.\\
Assme now $p=2$. Let $T \in U(\mathbb{Z}_2)$, one easily see that there is at least one coordinate of T belonging to ${\mathbb{Z}_2^\times}$.
A simple Hilbert symbol calculation implies the claim for $p=2$.\\
For $p=3$, note that  $ax^{2}+y^{2}+z^{2}-xyz=4a+3d^2$ is equivalent to the following equations
$$\begin{cases}
(2z-xy)^2-12d^2=(x^2-4)(y^2-4a)\\
(2ax-zy)^2-3d^2y^2=(z^2-4a-3d^2)(y^2-4a)
\end{cases}
$$
Then for any $P \in U(\mathbb{Z}_3)$, there are two coordinates of $P$ belonging to $3\mathbb{Z}_3$. We can assume $ x,y \in 3\mathbb{Z}_3 $,
since $(x^2-4,3)_3=(y^2-4a,3)_3=\frac{1}{2}$, one concludes that  $\inv_3 \mathcal B(P)= \frac{1}{2}$.
The proposition is established.
\end{proof}
\begin{proposition}
  Let  U be the scheme over $\Bbb Z$ given by
  \begin{equation}   ax^2+y^2+z^2 -xyz=4a+6d^2  \end{equation}
  where $4 \mid a, a \equiv 1 \bmod 3$ and $\sqrt{a} \notin Q$,
  $d \in \Bbb Z$ whose prime divisors are congruent to $\pm 1  \mod 12$ and $\pm 1  \mod 8$ or $\pm 5  \mod 12$ and $\pm 3  \mod 8$.
Then there is a Brauer-Manin obstruction to the integral Hasse principle for U.
 \end{proposition}
\begin{proof}
Note that if $U(\mathbb{Z}_2) \ne \emptyset$, since $4 \mid a$,for any local solution $T(x,y,z) \in U(\mathbb{Z}_2)$,  y or z is in ${\mathbb{Z}_2^\times}$.
We assume z is in ${\mathbb{Z}_2^\times}$,hence $z^2-4a \equiv 1 \bmod 8$. Thus $z^2-4a \in {Q_2^\times}^2$. Let $\mathcal{B}=(x^2-4,6)=(z^2-4a,6)=(y^2-4a,6)$,
we obtain $\inv_2 \mathcal B(T)=0$.\\
By lemma 3.5, we have $U(A_\mathbb Z) \ne \emptyset$. A similar argument in the proof of Proposition 3.6, we obtain
\begin{equation}
  \inv_p \mathcal B(T) =
  \begin{cases}
    1/2 &\mbox{if } p = 3, \\
    0  &\mbox{otherwise,}
  \end{cases}
\end{equation}
so that $\mathcal B$  gives an obstruction to the Hasse principle.\\
\end{proof}

\begin{proposition}
  Let  U be the scheme over $\Bbb Z$ given by
  \begin{equation}   ax^2+y^2+z^2 -xyz=4a+10d^2  \end{equation}
  where $a,d$ are odd integers such that $(a,d)=1$, $ord_5(a) \geq 2$ ,and $\sqrt{a} \notin Q$,
  the prime divisors of $d$ are congruent to $\pm 1  \mod 8$ and $\pm 1  \mod 5$ or $\pm 3  \mod 8$ and $\pm 2  \mod 5$.
   Then there is a Brauer-Manin obstruction to the integral Hasse principle for U.
 \end{proposition}
\begin{proof}
By lemma 3.5, one can prove $U(A_\mathbb Z) \ne \emptyset$. Let $\mathcal B=(x^2-4,10)=(z^2-4a,10)=(y^2-4a,10)$,
 we will show that for each point $T \in U(\mathbb{Z}_p)$, we have
\begin{equation}
  \inv_p \mathcal B(T) =
  \begin{cases}
    1/2 &\mbox{if } p = 2, \\
    0  &\mbox{otherwise,}
  \end{cases}
\end{equation}
so that $\mathcal B$  gives an obstruction to the Hasse principle.\\
If $p \nmid 10d^2$ the claim follows from Lemma 3.1 . If $p \mid d$, then $10 \in {Q_p^\times}^2$.
Thus $10d^2 \in {Q_p^\times}^2$. Finally, since $m - 4a > 0$, the claim is trivial for $p=\infty $ . It remains to examine $p = 2, 5$.\\
Assme now $p=5$.  Since $25 \mid a$,for any local solution $T(x,y,z) \in U(\mathbb{Z}_5)$,  y or z is in ${\mathbb{Z}_5^\times}$.
We assume $z$ is in $\mathbb{Z}_5^{\times}$,hence $z^2-4a \in Q_{5}^{{\times}^2}$.
we obtain $\inv_5$ $\mathcal B(T)=0$.\\
For $p=2$,for any local solution $T(x,y,z) \in U(\mathbb{Z}_2)$, there is at least one coordinate of T belonging to ${\mathbb{Z}_2^\times}$.
We assume z is in ${\mathbb{Z}_2^\times}$,hence $z^2-4a \equiv 5 \bmod 8$.
we obtain $\inv_2 \mathcal B(T)=1/2$.\\
so that $\mathcal B$  gives a obstruction to the Hasse principle.\\
\end{proof}

\begin{remark}
 We can take $m=4a+2qd^2$, where $q$ is an odd prime , one easily obtains similar conclusions.
\end{remark}
\begin{proposition}
  Let  U be the scheme over $\Bbb Z$ given by
  \begin{equation}   tq^2x^2+y^2+z^2 -xyz=4tq^2+2q^2d^2  \end{equation}
  where $q$ is an odd prime, t is an odd integer such that $3 \nmid (t-1)$ ,$\sqrt{t} \notin Q $,
  $(t,d)=1$ and the prime divisors of d are congruent to $\pm 1  \mod 8$.
Then there is a Brauer-Manin obstruction to the integral Hasse principle for U.
 \end{proposition}
\begin{proof}
By lemma 3.5, one can prove $U(A_\mathbb Z) \ne \emptyset$.Let $\mathcal B=(x^2-4,2)=(z^2-4a,2)=(y^2-4a,2)$,
 we will show that for each point $T \in U(\mathbb{Z}_p)$, we have
\begin{equation}
  \inv_p \mathcal B(T) =
  \begin{cases}
    1/2 &\mbox{if } p = 2, \\
    0  &\mbox{otherwise,}
  \end{cases}
\end{equation}
so that $\mathcal B$  gives an obstruction to the Hasse principle.\\
If $p \nmid 2q^2d^2$ the claim follows from lemma 3.1 . If $p \mid d$, then $2 \in {Q_p^\times}^2$.
Thus $2q^2d^2 \in {Q_p^\times}^2$. Finally, since $m - 4a > 0$, the claim is trivial for $p=\infty $ .
Note that for any local solution $T(x,y,z) \in U(\mathbb{Z}_2)$, there is at least one coordinate of T belonging to ${\mathbb{Z}_2^\times}$.
we obtain $\inv_2 \mathcal B(T)=1/2$. It remains to examine $p = q$.\\
If y or z is in ${\mathbb{Z}_q^\times}$, we have $y^2-4a \in {Q_q^\times}^2$ or $z^2-4a \in {Q_q^\times}^2$. If not,
for any point $M(x,y,z)\in U(\mathbb{Z}_q)$,let $y=qy',z=qz'$, then $(x, y',z')$ is the solution of the following equation
$$ t\mu_1^2+\mu_2^2+\mu_3^2 -\mu_1\mu_2\mu_3 =4t+2d^2$$
We obtain $\inv_q \mathcal B(M)=0$ by lemma 3.1 .
\end{proof}
\begin{proposition}
  Let  U be the scheme over $\Bbb Z$ given by
  \begin{equation}   -qx^2+y^2+z^2-xyz=-2q  \end{equation}
  where $q$ is an odd prime such that $ q \equiv \pm 3 \bmod 8$ ,
then there is a Brauer-Manin obstruction to the integral Hasse principle for U.
 \end{proposition}
\begin{proof}
  One can easily check $U(A_\mathbb Z) \ne \emptyset$ by lemma 3.5. Let $\mathcal B=(x^2-4,2q)=(z^2+4q,2q)=(y^2+4q,2q)$,
 we will show that for each point $T \in U(\mathbb{Z}_p)$, we have
\begin{equation}
  \inv_p \mathcal B(T) =
  \begin{cases}
    1/2 &\mbox{if } p = 2, \\
    0  &\mbox{otherwise,}
  \end{cases}
\end{equation}
We only need to consider $p=2,q,\infty.$
Note that for any local solution $T(x,y,z) \in U(\mathbb{Z}_2)$, there is at least one coordinate of T belonging to ${\mathbb{Z}_2^\times}$.
we obtain $\inv_2 \mathcal B(T)=1/2$. Since $ y^2+4q > 0$, the claim is trivial for $p=\infty $ . It remains to examine $p = q$.\\
If y or z is in ${\mathbb{Z}_q^\times}$, we have $y^2+4q \in {Q_q^\times}^2$ or $z^2+4q \in {Q_q^\times}^2$.If not,
for any point $M(x,y,z)\in U(\mathbb{Z}_q)$,let $y=qy',z=qz'$, then $(x, y',z')$ is the solution of the following equation
$$ -\mu_1^2+q\mu_2^2+q\mu_3^2 -q\mu_1\mu_2\mu_3 =-2$$
Thus $x^2 \equiv 2 \bmod q$, a contradiction. We obtain $\inv_q \mathcal B(M)=0$.
\end{proof}

\bigskip

\section {Review of bicyclic group cohomolgy}
Let $G=\mathbb{Z}/n\oplus \mathbb{Z}/m$, with generators $t$ and $s$.Put $N_t:=1+t+\cdots + t^{n-1}$ and $\Delta_t := 1-t$ in $\mathbb{Z}[G]$,
similar put $N_s:=1+s+\cdots + s^{m-1}$ and $\Delta_s := 1-s$ in $\mathbb{Z}[G]$.
For trivial G-module $\mathbb{Z}$ ,we have the following resolution
\begin{equation}
\cdots \mathbb{Z}[G]^4 \xrightarrow{d_2} \mathbb{Z}[G]^3\xrightarrow{d_1}\mathbb{Z}[G]^2\xrightarrow{d_0}
\mathbb{Z}[G].
\end{equation}
where $$ d_2=\begin{pmatrix}
\Delta_t& \Delta_s & 0 & 0 \\
0 & -N_t &N_s  & 0 \\
0 & 0 & \Delta_t &\Delta_s \end{pmatrix},\qquad
d_1=\begin{pmatrix}
N_t & \Delta_s & 0 \\
0 & -\Delta_t & N_s \end{pmatrix},\qquad
d_0=\begin{pmatrix}
\Delta_t & \Delta_s  \\
 \end{pmatrix},
$$
If we are given a G-module M,then applying $Hom_G(-,M)$ to the above complex,the
groups $H^i(G,M)$ are homology groups of the following complex:
$$M \xrightarrow{d^0} M^2 \xrightarrow{d^1} M^3 \xrightarrow{d^2} M^4 \cdots$$
where $$
d^0=\begin{pmatrix}
\Delta_t\\ \Delta_s  \\
 \end{pmatrix}, \qquad
 d^1=\begin{pmatrix}
 N_t & 0  \\
 \Delta_s& -\Delta_t \\
 0 & N_s   \end{pmatrix},\qquad
 d^2=\begin{pmatrix}
 \Delta_t &  0 & 0 \\
 \Delta_s & -N_t & 0 \\
 0 & N_s &\Delta_t\\
 0 & 0 & \Delta_s \end{pmatrix},$$\\

We introduce the notations: $Z^1(G,M):=ker(d^1)$,and $Z^2(G,M):=ker(d^2)$,then we have
$$ \begin{cases} Z^1(G,M)=\{(a,b)\in M^2|N_t(a)=N_s(b)=0, \Delta_s(a)=\Delta_t(b)\}\\
Z^2(G,M)=\{(a,b,c)|a\in M^t,c \in M^s,N_t(b)=\Delta_s(a), N_s(b)=-\Delta_t(c)\}
\end{cases} $$
For subgroup $\langle t \rangle$, we have the following resolution
\begin{equation}
\cdots \mathbb{Z}[t] \xrightarrow{\Delta_t} \mathbb{Z}[t]\xrightarrow{N_t}\mathbb{Z}[t]\xrightarrow{\Delta_t}
\mathbb{Z}[t].
\end{equation}
The injection from $\mathbb{Z}[t]$ to the first factor $\mathbb{Z}[G]$ of $\mathbb{Z}[G]^{i+1}$ induces
the restriction
\begin{center}
   $H^i(G,M) \to H^i(\langle t \rangle,M)$\\
  $(a_0,...a_i) \to  a_0$\\
\end{center}
Similar for subgroup $\langle s \rangle$, the injection from $\mathbb{Z}[s]$ to the last factor $\mathbb{Z}[G]$ of $\mathbb{Z}[G]^{i+1}$ induces
the restriction
\begin{center}
   $H^i(G,M) \to H^i(\langle s \rangle,M)$\\
  $(a_0,...a_i) \to  a_i$\\
\end{center}
\section {Special examples}

\begin{exam}
Let U be an affine variety over a field of characteristic zero defined by the equation
$$ ax^{2}+y^{2}+z^{2}-xyz=m$$
where $a \in k^{\times},a \notin k^2,m \ne 0 ,4a $.
By[4,Proposition 2.2],$Pic(\overline{U})$ is given by the following quotient group
$$ (\oplus^{6}_{i=1}\mathbb{Z}l_{i}\oplus \mathbb{Z}l)/(l-l_j-l_{j+3}:1\le j \le3)\cong \oplus^4_{i=1}\mathbb{Z}{\overline l_i}.$$\\
Here we give explicit condition which $H^1(k,Pic(\overline{U})) \cong \mathbb{Z}/2\oplus\mathbb{Z}/4$, and use methods of Colliot-Th\'el\`ene, D. Kanevsky, J.-J. Sansuc [5]
to describe the 4 torsion elements .\\

\begin{lemma}
  When $[k(\sqrt{a},\sqrt{m},\sqrt{m-4a}):k]=4$ and $\frac{\sqrt{m-4a}}{\sqrt{ma}}\in k$, $H^1(k,Pic(\overline{U})) \cong \mathbb{Z}/2\oplus\mathbb{Z}/4$
\end{lemma}
\begin{proof}
Let $G=Gal(k(\sqrt{a},\sqrt{m})/k)$,there exist $\sigma,\tau \in G $,such that
$$ \begin{cases}
\sigma(\sqrt{a})=-\sqrt{a},\ \ \sigma(\sqrt{m})=\sqrt{m},\ \ \sigma(\sqrt{m-4a})=-\sqrt{m-4a},\\
\tau(\sqrt{a})=-\sqrt{a},\ \ \tau(\sqrt{m})=-\sqrt{m},\ \ \tau(\sqrt{m-4a})=\sqrt{m-4a}.
\end{cases}
$$
By computation of proposition 2.3, we have \\
\begin{enumerate}[(i)]
\item $\left \{\begin{array}{l}
\sigma (\overline l_1)=-\overline{l_1} \\
\sigma (\overline l_2)= \overline l_3-\overline l_1 \\
\sigma (\overline l_3)= \overline l_2-\overline l_1 \\
\sigma (\overline l_4)= \overline l_2+\overline l_3-\overline l_1-\overline l_4\\
\end{array} \right.$
\\
\item $\left\{\begin{array}{l}
\tau (\overline l_1)=\overline l_1  \\
\tau (\overline l_2)=\overline l_1-\overline l_3 \\
\tau (\overline l_3)=\overline l_1-\overline l_2  \\
\tau (\overline l_4)=-\overline l_4\\
\end{array}\right.$
\end{enumerate}

Since $$
  Ker(1+\tau)=(\overline l_1-\overline l_3-\overline l_2,\overline l_4) ,\ \ \ \ \
  Im(\sigma-1)=(\overline l_1-\overline l_3-\overline l_2,2\overline l_4), \\
$$
we have $H^1(\langle \tau \rangle,Pic(\overline{U}))\cong \mathbb{Z}/2$.
By computation,we have $Pic(\overline{U})^{\langle \tau \rangle}=(\overline l_1,\overline l_2-\overline l_3)$, and since
$$\begin{cases}
\sigma(\overline l_1)=-\overline l_1,\\
\sigma(\overline l_2-\overline l_3)=-(\overline l_2-\overline l_3), \\
\end{cases}
$$
one concludes that
$$
H^1(\langle \sigma \rangle,Pic(\overline{U})^{\langle \tau \rangle})\cong \mathbb{Z}/2\oplus\mathbb{Z}/2,\quad
H^2(\langle \sigma \rangle,Pic(\overline{U})^{\langle \tau \rangle})=0.
$$
Hence,we have the following sequence
$$ 0\to \mathbb{Z}/2\oplus\mathbb{Z}/2 \to  H^1(G,Pic(\overline{U}))\to(\mathbb{Z}/2)^{\langle \sigma \rangle}\to 0$$
by [6.proposition 3.3.14]. To show $H^1(G,Pic(\overline{U}))$ has 4-torsion elements, we use bicyclic group cohomology .
Now we identify classes in $H^1(G,Pic(\overline{U}))$ with pairs $(a,b) \in Z^1(G,Pic(\overline{U}))$ modulo those of the form $(\Delta_{\sigma}(v),\Delta_{\tau}(v))$,
where
 $$Z^1(G,Pic(\overline{U}))=\{(a,b)\in Pic(\overline{U})^2|(1+\sigma)a=(1+\tau)b=0, \Delta_{\sigma}(b)= \Delta_{\tau}(a)\}.$$
Then any element of $H^1(G,Pic(\overline{U}))$ is the class of
$$ (x_1\overline l_1+x_2(\overline l_3-\overline l_1-\overline l_2)-y_2(\overline l_3-\overline l_4),y_1(\overline l_1-\overline l_2-\overline l_3)+y_2 \overline l_4)$$
where $x_1,x_2,y_1$ and $y_2$ $\in \mathbb{Z}$.
If we let $y_2$ be odd, it's easy to prove it's 4-torsion element.
\end{proof}

\bigskip

\begin{remark}
Using methods of Colliot-Th\'el\`ene,Dasheng Wei,and Fei Xu,we can obtain all 2-torsion elements:
$(x+2,m-4a)$,$\quad(x-2,m-4a)$,\quad$(x^2-4,m-4a)$
\end{remark}

\medskip

We take $x_1=2,x_2=1,y_1=0,y_2=1$,we obtain this class $(\overline l_1+\overline l_4-\overline l_2,\overline l_4)$,
since $\overline l_1+\overline l_4=\overline l_2+\overline l_5$ in $Pic(\overline{U})$,\
$(\overline l_5,\overline l_4)$ is a $4$-torsion element in $H^1(G,Pic(\overline{U}))$.Let $K=k(\sqrt{a},\sqrt{m})$, one has the following commutative diagram of exact sequences
$$
\xymatrix{
0\ar[r] & \Br(k,K) \ar[r]\ar[d]^{=} & \Br(U,K) \ar[r] \ar[d] &  H^1(G,Pic(U_K))  \ar[r]\ar[d]^{\partial} & 0\\
0 \ar[r] & H^2(G,K^{\times})  \ar[r]    & H^2(G,K(U)^{\times}) \ar[r]  &  H^2(G,K(U)^{\times}/K^{\times}) \ar[r] & 0}
$$
where the morphism $\partial$ is the connecting homomorphism of the following exact sequence
$$1 \to K(U)^{\times}/K^{\times} \to Div(U_K) \to Pic(U_K)\to 0$$
\medskip
$d^1(l_5,l_4)=(l_5(1,1)+l_5(1,-1),l_5(1,1)+l_4(-1,1)-l_5(-1,1)-l_4(1,1),l_4(1,1)+l_4(1,-1))\in Z^2(G,Div(U_K))$,
let $$\begin{cases}
f=\frac{1}{2}(\sqrt{m}-\sqrt{m-4a}-2\sqrt{a})xy+\sqrt{m-4a}y+(2\sqrt{a}-\sqrt{m})z-\sqrt{a}\sqrt{m-4a}x+\sqrt{m}\sqrt{m-4a}\\
g=\frac{1}{2}(-\sqrt{m}-\sqrt{m-4a}-2\sqrt{a})xy+\sqrt{m-4a}y+(2\sqrt{a}+\sqrt{m})z-\sqrt{a}\sqrt{m-4a}x-\sqrt{m}\sqrt{m-4a}
\end{cases}$$
By [1,proposition 7.1(b) and proposition 8.4],we have
$$\begin{cases}
div(f)=l_4(1,1)+l_5(-1,1)+l_1(1,-1)+l_2(-1,1)\\
div(g)=l_5(1,1)+l_4(-1,1)+l_1(1,-1)+l_2(-1,1)\\
\end{cases}$$
This implies that $div(\frac{g}{f})=l_5(1,1)+l_4(-1,1)-l_5(-1,1)-l_4(1,1)$\\
Since $div(y-\sqrt{m})=l_5(1,1)+l_5(1,-1)$, \quad $div(x-\frac{\sqrt{m}}{\sqrt{a}})=l_4(1,1)+l_4(1,-1))$, we obtain
$$ \partial((\overline l_5,\overline l_4))=(y-\sqrt{m},\frac{g}{f},x-\frac{\sqrt{m}}{\sqrt{a}}) \in Z^2(G,K(U)^{\times}/K^{\times})$$
Now we claim $(\sqrt{m}y-m,\frac{g(2\sqrt{a}-\sqrt{m}+\sqrt{m-4a})}{f(2\sqrt{a}+\sqrt{m}-\sqrt{m-4a})},x-\frac{\sqrt{m}}{\sqrt{a}})$ $\in Z^2(G,K(U)^{\times})$,
it suffices to show
$$ \begin{cases}
\sigma(\sqrt{m}y-m)=\sqrt{m}y-m\\
\tau(x-\frac{\sqrt{m}}{\sqrt{a}})=x-\frac{\sqrt{m}}{\sqrt{a}}\\
(1+\sigma)(\frac{g(2\sqrt{a}-\sqrt{m}+\sqrt{m-4a})}{f(2\sqrt{a}+\sqrt{m}-\sqrt{m-4a})})=(1-\tau)(\sqrt{m}y-m)\\
(1+\tau)(\frac{g(2\sqrt{a}-\sqrt{m}+\sqrt{m-4a})}{f(2\sqrt{a}+\sqrt{m}-\sqrt{m-4a})})=(\sigma-1)(x-\frac{\sqrt{m}}{\sqrt{a}})
\end{cases} $$
in $K(U)^{\times}$, one can directly check this by using rational point $(0,0,\sqrt{m})$ of $U_K$.The cocycle determines a non-cyclic Azumaya algebras
 $\mathcal A $ on U.
\end{exam}.\\
\begin{proposition}
 Let U be the affine scheme defined by
  $$ax^{2}+y^{2}+z^{2}-xyz=m$$
 where $a,m \in \mathbb{Z}$. Let $K=Q(\sqrt{m},\sqrt{a})$, when\\
 \begin{enumerate}[(i)]
 \item $[K:Q]=4$, $\frac{\sqrt{m-4a}}{\sqrt{ma}}\in Q$ \\
 \item for any prime q,its decomposition group in $Gal(K/Q)$ is cyclic\\
 \item There exists a prime $p \geq 5$,such that p splits in $Q(\sqrt{m})$ and has ramification index 2 in $Q(\sqrt{a})$\\
 \end{enumerate}
then there is no Brauer-Manin obstruction to the integral Hasse principle for U.
\end{proposition}

\begin{proof}
We can assume that $U(A_\mathbb Z) \ne \emptyset$, where $A_\mathbb Z=\RR \times \prod_p \mathbb{Z}_p$ , otherwise there is nothing to prove. Note that
since p splits in $Q(\sqrt{m})$, its decomposition group is $\langle \sigma \rangle$. Hence for any $T \in U(\mathbb{Z}_p)$,
$\mathcal A(T)=(\sqrt{m}y-m,a) \in \Br(Q_p)$.By (iii), we can assume $ord_{p}(a)=1$. Note that
$$ \begin{cases}
(\sqrt{m}y-m,a)_{p}+(y-\sqrt{m},a)_{p}=(\sqrt{m},a)_{p}\\
(x+2,m-4a)_{p}=(x+2,ma)_{p}=(x+2,a)_{p}\\
\end{cases}
$$
Let $\mathcal B_1=(x+2,a)$, $\mathcal B_2=(y-\sqrt{m},a)$,to prove the proposition ,
it suffices to show for all $(\varepsilon_1,\varepsilon_2) \in (\mathbb{Z}/2\mathbb{Z})^{2}$, there exists
$\zeta \in U(\mathbb{Z}_p)$ such that
$$ (inv_{p}\mathcal B_1(\zeta),inv_{p}\mathcal B_2(\zeta))=(\varepsilon_1,\varepsilon_2).$$
Since $ord_p(a)$ is odd, we have p $\mid m$ by (i) ,in fact $ord_{p}(m)$ is even by (iii).
Let $s,t\in \Bbb F_{p}^{\times}$, such that the Legendre symbol $(\frac{s}{p})=1$, $ (\frac{t}{p})=-1$, let
$$ \upsilon_1=(2,s,s),\quad \upsilon_2=(2,t,t). $$
One can see that $\upsilon_1$ and $\upsilon_2$ are smooth points of $U(\Bbb F_{p})$, hence there exist $\mu_{1},\mu_{2} \in U(\mathbb{Z}_p)$
such that $\mu_{i}\equiv \upsilon_{i} \bmod p$ for $1 \leq i\leq 2$, we obtain
$$ \begin{cases}
(inv_{p}\mathcal B_1(\mu_1),inv_{p}\mathcal B_2(\mu_1))=(0,0)\\
(inv_{p}\mathcal B_1(\mu_2),inv_{p}\mathcal B_2(\mu_2))=(0,\frac{1}{2})\\
\end{cases}
$$
By Dirichlet theorem,there exist a prime $l$,such that Legendre symbol $(\frac{l}{p})=-1$ and $p \nmid (l+1)$. Let
$$\upsilon_3=(\frac{l^{2}+1}{l},s,ls),\quad \upsilon_4=(\frac{l^{2}+1}{l},t,lt). $$
One can check that $\upsilon_3$ and $\upsilon_4$ are smooth points of $U(\Bbb F_{p})$, hence there exist $\mu_{3},\mu_{4} \in U(\mathbb{Z}_p)$
such that $\mu_{i}\equiv \upsilon_{i} \bmod p$ for $3 \leq i\leq 4$, we obtain
$$ \begin{cases}
(inv_{p}\mathcal B_1(\mu_3),inv_{p}\mathcal B_2(\mu_3))=(\frac{1}{2},0)\\
(inv_{p}\mathcal B_1(\mu_4),inv_{p}\mathcal B_2(\mu_4))=(\frac{1}{2},\frac{1}{2})\\
\end{cases}
$$
The proposition is established.
\end{proof}

\subsection*{Acknowledgements}
I thank my advisor Fei Xu for countless helpful conversations and suggestions, and also for his patience.  I also thank Yang Zhang for many useful discussions.

\bibliographystyle{alpha}
\end{document}